\documentclass[a4paper,11pt]{article}
\usepackage{amsmath}
\usepackage{amssymb,bm}
\usepackage{amsbsy}
\usepackage{amsfonts,ntheorem}
\usepackage{graphicx}
\usepackage{sectsty}
\usepackage{verbatim}
\usepackage{enumitem}
\usepackage{ifthen}
\usepackage{tikz}
\usepackage{psfrag}
\usepackage[titletoc]{appendix}
\usepackage{color}
\usepackage{microtype}\usepackage{aliascnt}
\usepackage{placeins}

\usepackage[pdftex,colorlinks=true,linkcolor=blue,citecolor=green,urlcolor=blue,bookmarks=true,bookmarksnumbered=true]{hyperref}

\usepackage{bbm}

\newcommand{\R}{\mathbbm{R}}

\newcommand{\ve}{\varepsilon}
\def\d{d}

\DeclareMathOperator{\dist}{dist}

\DeclareMathOperator{\sgn}{sgn}

\newcommand{\abs}[1]{\lvert#1\rvert}
\newcommand{\norm}[1]{\left\|#1\right\|}
\newcommand{\set}[1]{\left\{#1\right\}}
\newcommand{\inner}[2]{\left<#1,#2\right>}
\newcommand{\supp}{\mathop{\mathrm{supp}}}

\newcommand{\cone}{\mathcal{K}}

\newtheorem{lemma}{Lemma}[section]

\newaliascnt{proposition}{lemma}

\aliascntresetthe{proposition}

\newaliascnt{theorem}{lemma}
\newtheorem{theorem}[theorem]{Theorem}
\aliascntresetthe{theorem}

\newaliascnt{corollary}{lemma}
\newtheorem{corollary}[corollary]{Corollary}
\aliascntresetthe{corollary}

\newaliascnt{invpro}{lemma}

\aliascntresetthe{invpro}

\theorembodyfont{\normalfont}
\newaliascnt{definition}{lemma}
\newtheorem{definition}[definition]{Definition}
\aliascntresetthe{definition}

\theorembodyfont{\normalfont\itshape}
\theoremseparator{:}
\theoremheaderfont{\kern-1em\normalfont\itshape}
\newaliascnt{remark}{lemma}
\newtheorem{remark}[remark]{Remark}
\aliascntresetthe{remark}

\newenvironment{proof}{{\textit{Proof:}}}{\hfill \textbf{$\square$}\vspace{0.2cm}}

\definecolor{darkgreen}{rgb}{.0,.25,.0}

\title{Photoacoustic Tomography With Spatially Varying Compressibility and Density}

\author{Zakaria Belhachmi
\thanks{Laboratoire de Math\'ematiques LMIA, Universit\'e de Haute Alsace, 4, rue des Fr\`eres Lumi\`ere,
68200 Mulhouse, France.(\tt{zakaria.belhachmi@uha.fr})} \and
Thomas Glatz
\thanks{Computational Science Center, University of Vienna, Oskar-Morgenstern Platz 1, A-1090 Vienna, Austria.
(\tt{thomas.glatzl@univie.ac.at})}\and
Otmar Scherzer
\thanks{Computational Science Center, University of Vienna, Oskar-Morgenstern Platz 1, A-1090 Vienna, Austria,
and Johann Radon Institute for Computational and Applied Mathematics (RICAM),
Austrian Academy of Sciences, Altenbergerstra\ss{}e 69,
A-4040 Linz, Austria.(\tt{otmar.scherzer@univie.ac.at})}
}

\begin{document}

\maketitle

\begin{abstract}
This paper investigates photoacoustic tomography with two spatially varying acoustic parameters, the compressibility 
and the density. We consider the reconstruction of the absorption density parameter (imaging parameter of Photoacoustics)
with complete and partial measurement data. 
We investigate and analyze three different numerical methods for solving the imaging problem and compare the results. 
\end{abstract}
\noindent
\textbf{Keywords:} Photoacoustic imaging, spatially varying compressibility and density, variable sound speed, regularization,
time reversal.

\section{Introduction}
Photoacoustic Imaging (PAI) is a novel technique for tomographical imaging of 
small biological or medical specimens.
The method makes use of the fact that an object expands after being exposed to 
ultrashort electromagnetic radiation, and emits an ultrasonic wave (see e.g. \cite{XuWan06,Wan09}). 
The resulting acoustic pressure is assumed to be proportional to the electromagnetic \emph{absorption},
which is the imaging parameter of Photoacoustics. It provides detailed anatomical and functional 
information. 

Opposed to the conventional photoacoustic imaging \cite{XuWan06,Wan09}, which is based on the assumption that 
the \emph{compressibility} and \emph{density} of the medium are constant (and thus in turn the sound speed),
this paper assumes \emph{both} of these parameters spatially varying.
The mathematical model describing the propagation of the ultrasonic pressure considered here is 
\begin{equation}\label{eq:wave_intro}
\begin{aligned}
\kappa(x)y''(x,t) - \nabla \cdot  \left(\rho(x)^{-1}\nabla y(x,t) \right) &=0 \text{ in } \R^n  \times (0,\infty)\\
        y(x,0) &= f(x) \text{ in } \R^n\\
        y'(x,0) &= 0 \text{ in } \R^n.
\end{aligned}
\end{equation}
Here, $\kappa$ is the material compressibility, $\rho$ denotes the density and $f$ denotes the amount of absorbed energy, 
i.e. the imaging parameter that encodes the material properties of physiological interest in PAI. 
We emphasize that the 
speed of sound is given by $c(x)=\frac1{\sqrt{\kappa\rho}}$ and that this equation is more general than
\begin{equation}\label{eq:wave_old}
\begin{aligned}
y''(x,t) - c^2(x) \Delta y(x,t) & = 0 \text{ in } \R^n  \times (0,\infty)\\
        y(x,0) & = f(x)  \text{ in } \R^n\\
        y'(x,0) &= 0 \text{ in } \R^n,
\end{aligned}
\end{equation}
which also describes acoustic wave propagation in the case of variable sound speed. The latter equation is derived from \autoref{eq:wave_intro} under the additional assumption that $\rho$ is spatially slowly varying.
For further details on the derivation of \autoref{eq:wave_intro} from fluid- and thermodynamics, we refer to \cite[Chapter 8.1]{ColKre92}.

The \emph{photoacoustic reconstruction} consists in determining the function $f$ from 
measurement data of $y$ on a surface $S\subseteq \partial \Omega$ over time $(0,T)$.

There exists a huge amount of literature on reconstruction formulas in the case $c \equiv 1$, see for instance 
\cite{XuXuWan02,XuWan02a,XuWan02b,XuXuWan03,XuWan05,KucKun08,Kuc14} to name but a few. Time reversal in the case of variable 
sound speed has been studied for instance in \cite{AgrKuc07,HriKucNgu08,SteUhl09,QiaSteUhlZha11,TreVarZhaLauBea11,HuaNieSchWanAna12}.
Time reversal for photoacoustic imaging based on \autoref{eq:wave_intro} as well as on \autoref{eq:wave_old} has been given in \cite{SteUhl09} - note that both associated wave operators are special cases of the general operator from \cite{SteUhl09}.
Their theory has been generalized to the elastic wave equation in \cite{Tit12}.

In this paper we focus on numerical realization and regularization theory of photoacoustic imaging based 
on \autoref{eq:wave_intro} with different numerical methods. Most closely related to our numerical approach 
is a time reversal algorithm from \cite{TreCox10}, which employs the formula \autoref{eq:TR_TC} below. 
Recently we applied iterative regularization techniques in the case of variable sound speed \cite{BelGlaSch15_report} 
and compared it with time reversal. The goal here is to generalize time reversal and iterative regularization for photoacoustic 
imaging in the case of spatially variable density and compressibility. A convergence in the least-squares-sense is thereby guaranteed by standard results from regularization theory (see e.g. \cite{Gro84,EngHanNeu96}).

The paper is organized as follows: In \autoref{sec:direct} we analyze the mathematical equations describing 
wave propagation in the case of spatially variable compressibility and density. 
Imaging based on this model is analyzed in \autoref{sec:pai}. 
Numerical results are presented for three different methods; Time reversal, Neumann series, and Landweber iteration 
in \autoref{sec:results}. The latter two seem to be new for the presented equation.
We also investigate the case of partial measurement data.

%

\subsection{Notation}
In the beginning we summarize the basic notation, which is used throughout the paper.

$\Omega$ denotes a non-empty, open, bounded and connected domain in $\R^n$ with Lipschitz and piecewise 
$C^1$-boundary $\partial \Omega$. Moreover, $S \subseteq \partial \Omega$ is connected and relatively open. 
The vector $\bm n(x)$, with $x\in \partial \Omega$, denotes the outward pointing unit normal vector of $\Omega$.
The absorption density $f$, the compressibility $\kappa$ and the density $\rho$ are supposed to satisfy: 
\begin{itemize}
\item $\kappa,\rho \in C^1(\R^n)$, satisfying $0 < \rho_{\min} \leq \rho(x) \leq \rho_{\max}$, 
        $0 < \kappa_{\min} \leq \kappa(x) \leq \kappa_{\max}$. We also define $c_{\max}:=(\kappa_{\min}\rho_{\min})^{-1/2}$.
       Moreover, we assume that $\kappa,\rho$ are constant in $\mathcal{C}\Omega$ and satisfy 
       $\kappa \rho = 1$ there. 
 \item The absorption density function $f$ has support in $\Omega$: $\text{supp}(f) \subseteq \Omega$.
\end{itemize}
For the sake of simplicity of notation we omit space and time arguments of functions whenever this is convenient and does 
not lead to confusions.

We use the following Hilbert spaces:
\begin{itemize}
\item We denote by $L^2(\Omega) = \set{\phi \in L^2(\R^n) : \phi \equiv 0 \text{ in } \mathcal{C}\Omega}$, where $\mathcal{C}\Omega$
      denotes the complement of $\Omega$, with inner product
      \begin{equation*}
      \inner{\phi_1}{\phi_2}_{L^2(\Omega)}  = \int_{\R^n} \phi_1 \phi_2  \d x.
      \end{equation*}
\item For $\hat{\Omega} = \Omega$ or $\R^n$:
      \begin{itemize}
      \item Let $H_0^1(\hat{\Omega})$ be the closure of differentiable functions on $\R^n$ with compact support in 
            $\hat{\Omega}$, associated with the non-standard (but equivalent) inner product
            \begin{equation}
             \label{eq:kappa_rho}
             \inner{\phi_1}{\phi_2}_{1;\kappa,\rho} = \int_{\R^n}\kappa\phi_1\phi_2+ \rho^{-1}\nabla \phi_1\cdot \nabla \phi_2 \d x.
            \end{equation}
            The associated norm is denoted by $\norm{\phi}_{1;\kappa,\rho}$.
      \item The seminorm associated to the inner product
      		\begin{equation}\label{eq:rho}
      		\inner{\phi_1}{\phi_2}_{1;\rho} = \int_{\mathbb{R}^n}\rho^{-1}\nabla\phi_1\cdot\nabla \phi_2 dx
      		\end{equation}
      		is denoted by $\abs{\phi}_{1;\rho}$.
      \item The norm associated with the inner product
            \begin{equation*}
            \inner{\phi_1}{\phi_2}_1 = \int_{\hat{\Omega}} \phi_1 \phi_2 + \nabla \phi_1\cdot \nabla \phi_2 \d x
            \end{equation*}
            is denoted by $\norm{\phi}_1$.  
      \item The norms $\norm{\phi}_1$ and $\norm{\phi}_{1;\kappa,\rho}$ are equivalent. In fact, we have
      \begin{equation}\label{eq:equiv2}
      \begin{aligned}
      & \sqrt{\min\set{\kappa_{\min},{\rho_{\max}^{-1}}}}
      \norm{\phi}_1 \leq \norm{\phi}_{1;\kappa,\rho} 
      \leq \sqrt{\max\set{\kappa_{\max},{\rho_{\min}^{-1}}}} \norm{\phi}_1\,,\\
      & \qquad \text{ for all } \phi \in H_0^1(\hat\Omega).
      \end{aligned}
      \end{equation}
     \end{itemize}
\item $L^2(S)$ denotes the standard Hilbert space of square integrable functions on $\partial \Omega$ with support in $S$, 
      together with the inner product 
      \begin{equation*}
            \inner{\phi_1}{\phi_2}_{L^2(S)} = \int_S \phi_1(x) \phi_2(x) \d S(x).
      \end{equation*}
      $L^2(\Sigma)$ denotes the standard Hilbert space of square integrable functions on $\partial \Omega \times (0,T)$ with 
      support in $\Sigma :=  S  \times (0,T)$, together with the inner product 
      \begin{equation*}
            \inner{\phi_1}{\phi_2}_{L^2(\Sigma)} = \int_0^T \int_S \phi_1(x,t) \phi_2(x,t) \d S(x) dt.
      \end{equation*}
      The induced norms are denoted by $\norm{\cdot}_{L^2(S)}$, $\norm{\cdot}_{L^2(\Sigma)}$. 
      
      $C_0^\infty(\Sigma)$ denotes the space of smooth functions with support in $\Sigma$.
\item The trace operator 
      $\gamma_{\Omega} : H^1(\R^n) \to L^2(\partial \Omega)$ restricts functions defined on $\R^n$ 
      onto $\partial \Omega$, respectively. 
      This operator is the composition of the standard trace operator $\gamma: H^1(\Omega) \to L^2(\partial \Omega)$
      and the restriction operator $R: H^1(\R^n) \to H^1(\Omega)$, which are both bounded \cite[Theorem 5.22]{Ada75}, and thus 
      itself bounded. 
      We denote 
      \begin{equation}
       \label{eq:norm}
       \mathcal{C}_\gamma:=\norm{\gamma \circ R}.
      \end{equation}
\end{itemize}

\section{Direct Problem of Wave Propagation}
\label{sec:direct}
In this section we are analyzing the \emph{wave operator} $L$ mapping the absorption density $f$ onto the solution 
$y$ of \autoref{eq:wave_intro} restricted to $\Sigma$:
\begin{equation}\label{obs_op}
L: H_0^1(\Omega)  \rightarrow L^2(\Sigma), \quad  f \mapsto y|_\Sigma.
\end{equation}
First, we show that the operator $L$ is bounded. 
Analogous to \cite{BelGlaSch15_report}, we define the total wave energy by
\begin{align}\label{eq:wave_energy}
E(t;y) := \frac{1}{2} \int_{\R^n} \kappa y'{}^2 + \rho^{-1}\left| \nabla y\right|^2 \d x.
\end{align}
The time derivative of $E$, taking into account \autoref{eq:wave_intro}, is 
\begin{equation*}
E'(t;y) = \int_{\R^n}\kappa y'' y' - \nabla\cdot\left(\rho^{-1}\nabla y\right) y' \d x = 0 \text{ for all } t >0 .
\end{equation*}
Consequently,
\begin{align}
\label{eq:E}
E(t;y) = E(0;y) = \frac12 \abs{f}_{1;\rho}^2\text{ for all } t >0.
\end{align}
This, together with \autoref{eq:wave_energy} shows that
\begin{align}
\label{eq:wave_bound}
\int_{\R^n}\kappa  y'{}^2 \d x \leq \abs{f}_{1;\rho}^2
\text{ and }
\abs{y(t)}_{1;\rho} \leq \abs{f}_{1;\rho} 
\text{ for all } t >0.
\end{align}
A-priori this inequality does not provide a bound for $y$ in the whole $\R^n$ with respect to the standard $H^1$-norm.
This is provided for instance by the following lemma:
\begin{lemma}
 Let $y$ be the solution of \autoref{eq:wave_intro}, then 
 \begin{equation}
 \label{eq:fund_const}
  \norm{y(t)}_{1;\kappa,\rho} \leq \mathcal{C}(T)\norm{f}_{1;\kappa,\rho},\qquad 
  \text{ for all } t \in (0,T),
 \end{equation}
 where
 \begin{equation*}
  \mathcal{C}(T):= \sqrt{\max\set{1+2T^2,2}}.
 \end{equation*}
\end{lemma}
\begin{proof}
For arbitrary $\hat{t} \in (0,T)$ it follows from \autoref{eq:wave_bound} that:
\begin{equation*}
 \begin{aligned}
   ~ & \int_{\R^n}\kappa(x) (y(x,t)-y(x,0))^2 \d x = \int_{\R^n}\kappa(x) \left( \int_0^{\hat{t}} y'(x,\hat{t}) d\hat{t} \right)^2 \d x\\
\leq & \hat{t} \int_0^t \int_{\R^n}\kappa y'{}^2  \d x d\hat{t}
  \leq T^2 \abs{f}_{1;\rho}^2.
 \end{aligned}
\end{equation*}
With the elementary inequality $(a-b)^2 \geq \frac12 a^2 - b^2$ it follows that 
\begin{equation*}
 \begin{aligned}
   ~ & \int_{\R^n}\kappa(x) (y(x,t))^2 \d x \\
\leq & 2 \int_{\R^n}\kappa(x) (y(x,t)-y(x,0))^2 \d x + 2\int_{\R^n}\kappa(x) (y(x,0))^2 \d x\\
\leq & 2T^2 \abs{f}_{1;\rho}^2 + 2 \norm{\sqrt{\kappa} f}_{L^2(\mathbb R^n)}^2.
 \end{aligned}
\end{equation*}
This, together with \autoref{eq:wave_bound} shows the assertion.
\end{proof}

Now, we prove boundedness of $L$:
\begin{theorem}
The operator $L : H_0^1(\Omega) \to L^2(\Sigma)$ is bounded and 
\begin{equation} \label{eq:norm_l}
\norm{L} \leq \frac{\mathcal{C}_\gamma \mathcal{C}(T)\sqrt{T}}{\sqrt{\min\set{\kappa_{\min},{\rho_{\max}^{-1}}}}}.
\end{equation}
\end{theorem}
\begin{proof}
For given $f$ let $y$ be the solution of \autoref{eq:wave_intro}. 
From \autoref{eq:wave_energy} it follows that the solution $y$ of \autoref{eq:fund_const} is in $H^1(\R^n)$ for every $t > 0$.
Thus from \autoref{eq:norm} and \autoref{eq:fund_const} it follows that 
\begin{align*}
\norm{y}_{L^2(\Sigma)}^2 = \int_0^T \int_{\partial \Omega} y^2(t) d S(x) d t
\leq \mathcal{C}_\gamma^2 \int_0^T \norm{y(t)}_{1}^2 d t 
\leq \frac{\mathcal{C}_\gamma^2 \mathcal{C}(T)^2 T \norm{f}_{1;\kappa,\rho}^2}{\min\set{\kappa_{\min},{\rho_{\max}^{-1}}}}.
\end{align*}
\end{proof}

\begin{remark}\label{rem:opH1}
From \cite[Theorem 1]{Tat98} it follows that the trace $L[f]$ is in fact in $H^{3/4}_\text{loc}\left(\partial\Omega\times(0,T)\right)$. Therefore, by finite speed of propagation (\autoref{thm:prop_speed}), if we in addition assume that $f$ is supported away from $\partial\Omega$, $L[f]$ is in fact in $H^{3/4}\left(\partial\Omega\times(0,T)\right)$.
The proof does not follow in a straightforward way from standard trace results, but utilizes the theory of Fourier integral operators and microlocal analysis.
In special cases, this result can be further improved, see \cite[Remark 5]{SteUhl09} and also \cite{BaoSym91,BaoSym93}. For instance, 
for $\Omega$ strictly convex, we have $L[f]\in H^1\left(\partial\Omega\times(0,T)\right)$.
As a consequence, $L:H^1_0(\Omega)\to L^2(\Sigma)$ is compact, because it is a composition of bounded and a compact operator, 
using the fact that the embedding from $H^1(\partial\Omega\times(0,T))$ to $L^2(\partial\Omega\times(0,T))$ is compact 
(see e.g. \cite[Theorem 2.34]{Aub82}). 
\end{remark}

In the following we characterize the adjoint of $L: H_0^1(\Omega) \to L^2(\Sigma)$ on a dense subset of $L^2(\Sigma)$.
Since $L^*: L^2(\Sigma) \to H_0^1(\Omega)$ is bounded (in fact, $\norm{L}=\norm{L^*}$), a characterization 
follows by limits of convergent sequences on dense subsets. Here we characterize $L^*$ on $C_0^\infty(\Sigma)$
first and extend it by convergent sequences to $L^2(\Sigma)$. 
\begin{definition}
Let $i$ be the embedding operator from $H_0^1(\Omega)$ to $L^2(\Omega)$. 
Then $i^*: L^2(\Omega) \to H_0^1(\Omega)$ is the operator 
which maps a function $\psi \in L^2(\Omega)$ onto the solution of the equation
\begin{equation}\label{eq:ell1}
D u = \psi\text{ in } \Omega, \qquad u=0 \text{ on } \partial \Omega,
\end{equation}
with $D:=\kappa -\nabla\cdot\rho^{-1}\nabla$.
In other words $i^* = D^{-1}$.
\end{definition}

In the following we derive the adjoint $L^*$ of the operator $L$, which is required for 
the implementation of the Landweber iteration below, on a dense subset of $L^2(\Sigma)$.
The boundedness of $L^*$ is guaranteed by elementary Hilbert space theory. Therefore, we get a 
characterization on $L^2(\Sigma)$ by limits of convergent sequences.
\begin{theorem}
For $h \in C_0^\infty(\Sigma)$ the adjoint of the operator $L$, defined in \autoref{obs_op}, is given by
\begin{equation}\label{eq:adjoint0}
L^*[h] = i^* \circ L_D^*[h]
\end{equation}
where  
\begin{equation}
\label{eq:adjoint}
L_D^*[h] = \left. \kappa z'(0)\right|_\Omega,
\end{equation}
and $z:=z(h)$ is the weak solution of
 \begin{equation}\label{eq:wave_adj}
 \begin{aligned}
         \kappa z'' - \nabla\cdot(\rho^{-1}\nabla z) &=0 \text{ in } \R^n \backslash \partial \Omega \times (0,T),\\
          z(T) = z'(T) &=0  \text{ in } \R^n, \\          
          \left[ z \right] =0, \quad
        \left[  \frac{\partial z}{\partial \bm n} \right] &= h \text{ on } \Sigma,\\
        \left[ z \right] =0, \quad
        \left[  \frac{\partial z}{\partial \bm n} \right] &= 0 \text{ on }  \partial \Omega\setminus \Sigma\times(0,T).        
\end{aligned}
\end{equation}
Here 
\begin{equation*} [z]:=z^+|_\Sigma-z^-|_\Sigma \text{ and }
   \left[   \frac{\partial z}{\partial \bm n}\right] := 
   \left.\frac{\partial z^+}{\partial \bm n}\right|_\Sigma - \left.\frac{\partial z^-}{\partial \bm n} \right|_{\Sigma}
   \end{equation*}
where $z^+:=z|_{\mathcal{C}\Omega\times(0,T)}$ and $z^-:=z|_{\Omega\times(0,T)}$.
\end{theorem}
\begin{proof}
The existence of a weak solution of \autoref{eq:wave_adj} is analogous to \cite[Appendix A]{BelGlaSch15_report}, where 
the transmission problem has been studied for variable sound speed.
Multiplying the first equation of \autoref{eq:wave_adj} by the solution of \autoref{eq:wave_intro} $y$, it follows that 
$h \in C^\infty_0(\Sigma)$:
\begin{equation}
\label{eq:adjoint_proof}
\begin{aligned} 
~& \int_\Sigma h L[f] \d S(x) dt = \int_\Sigma h y \d S(x) dt = \int_\Omega \kappa z'(0) f \d x \\
=& \int_\Omega D \left[D^{-1} \left[\kappa z'(0)\right]\right] f \d x = 
 \inner{i^* \left[\kappa {z'(0)}\right]}{f}_{1;\kappa,\rho}.
\end{aligned}
\end{equation}
\end{proof}

The following results on finite speed of propagation of the solution of \autoref{eq:wave_intro} 
are based on the results from \cite{Eva10}.
\begin{definition}
Let ${\tt d}(x):= \dist(x_0,x)$ be the distance between $x_0$ and $x$ with respect to the Riemannian metric $\kappa\rho \d x$. 
Note that this metric is chosen in accordance with the principal symbol of the elliptic operator
\begin{equation*}
 -\kappa^{-1}\nabla\cdot\left(\rho^{-1}\nabla\,\cdot\,\right)\;.
\end{equation*}
The cone $\cone$ with respect to the space-time point $(t_0,x_0)$ is defined as 
\begin{equation*}
\cone:=\set{(x,t):{\tt d}(x)< t_0-t}.
\end{equation*}
we further introduce
\begin{equation*}\cone_t:=\set{x:{\tt d}(x)<t_0-t},
\end{equation*}
the cross-section of $\cone$ at fixed time $t$.
\end{definition}
\begin{remark}
According to \cite[pp. 416]{Eva10} the function ${\tt d}$ solves the PDE
\begin{align}\label{eq:HamJac}
\frac{1}{\rho \kappa}|\nabla {\tt d}|^2 = 1, \quad 
{\tt d}(x_0) = 0.
\end{align}
This relation will be used in the proof of \autoref{thm:prop_speed} below.
\end{remark}

\begin{theorem}\label{thm:prop_speed}
Let $y$ be the solution of \autoref{eq:wave_intro}.
Assume $y' \equiv 0$ on $\cone_0$, then $y\equiv 0$ in $\cone$.
\end{theorem}
\begin{proof}
Let
 \begin{align}\label{eq:local_energy}
 E_{\cone_t}(t;y):=\frac{1}{2}\int_{\cone_t} \kappa y'{}^2+\rho^{-1}\abs{\nabla y}^2 \d x
\end{align}
denote the local energy. 
Taking the time derivative of $E_{\cone_t}$, it follows by application of the Leibniz rule for the differentiation of evolving 
region integrals \cite[p.713]{Eva10} that 
\begin{equation*} E_{\cone_t}'(t;y)=\int_{\cone_t}\kappa y'y''+ \rho^{-1}\nabla y\cdot\nabla y' \d y +
   \frac12\int_{\partial \cone_t} \left(\kappa y'{}^2+\rho^{-1}\abs{\nabla y}^2\right)\bm v\cdot\bm n  \d S(x),
\end{equation*}
where $\bm v$ denotes the normal velocity of the moving boundary $\partial C_t$ and 
$\bm n$ is the outward pointing unit normal to $\cone_t$.
For a fixed point $x\in \cone_t$, we have  $\bm v(x)=\frac{\nabla {\tt d(x)}}{|\nabla {\tt d(x)}|^2}$ and $\bm n(x) = -\frac{\nabla {\tt d(x)}}{|\nabla {\tt d(x)}|}$.
Therefore,
\begin{equation}
 \label{eq:**}
\begin{aligned}
E_{\cone_t}'(t;y)&=\int_{\cone_t}\kappa y'y''+\rho^{-1}\nabla y\cdot\nabla y' \d y \\
             &\qquad -\frac12\int_{\partial \cone_t}\left(\kappa y'{}^2+\rho^{-1}\abs{\nabla y}^2\right)\frac{1}{|\nabla {\tt d}|}  \d S(x)=:A-B\;.
\end{aligned}
\end{equation}
By application of Green's formula, $A$ is rewritten to 
\begin{align*}
A=\int_{\cone_t}y'\left(\kappa y''-\nabla\cdot\left(\rho^{-1}\nabla y\right)\right)dx + 
  \int_{\partial \cone_t}\rho^{-1}y'\nabla y\cdot\frac{\nabla {\tt d}}{|\nabla {\tt d}|} \d S(x).
\end{align*}
The first term vanishes since $y$ is solution to \autoref{eq:wave_intro} and thus from \autoref{eq:HamJac} it follows that
\begin{align*}
|A| \leq 
\int_{\partial \cone_t} \frac{\abs{y'}\rho^{-1/2}\abs{\nabla {\tt d}}\abs{\nabla y}\rho^{-1/2}}{\abs{\nabla {\tt d}}} \d S(x)=
\int_{\partial \cone_t} \frac{\abs{y'}\kappa^{1/2}\abs{\nabla y}\rho^{-1/2}}{\abs{\nabla {\tt d}}} \d S(x)\;.
\end{align*}
Using Cauchy's inequality, we finally obtain \
\begin{align*}
|A|\leq \frac12\int_{\partial \cone_t}\left(\kappa y'{}^2+\rho^{-1}\abs{\nabla y}^2\right)\frac1{\abs{\nabla {\tt d}}}\d S(x) =B\;.
\end{align*}
Thus, from \autoref{eq:**} it follows that $E_{\cone_t}'(t;y)\leq 0.$
\end{proof}

\section{Photoacoustic Imaging}
\label{sec:pai}
The \emph{photoacoustic imaging problem} can be expressed as the solution of the operator equation:
\begin{equation}
\label{eq:operator}
 L[f] = m.
\end{equation}
Uniqueness of the solution of \autoref{eq:operator} is closely related to a unique continuation result 
for the wave equation:
\begin{theorem}[Unique Continuation \cite{Rob91,Hoe92,Tat95}]\label{thm:unique_cont}
Let $u$ be a solution to \autoref{eq:wave_intro}, and $y \equiv 0$ in $(-T,T)\times B_\ve(x_0)$. 
Then 
\begin{equation*}
 y\equiv 0 \text{ in } \set{(x,t)\in \R\times\R^n: \dist(x,x_0)\leq T-|t|}\;.
\end{equation*}
\end{theorem}

Uniqueness of the solution of \autoref{eq:operator} requires a sufficiently large observation time
\begin{align}\label{eq:T0}
T_0:=\max\limits_{x\in\supp f} \dist(x,S),
\end{align}
where $\dist(x,S)$ is the distance of $x$ to the closest point $x'\in S$
with respect to the Riemannian metric $\kappa\rho \d x$. 
\begin{corollary}[Injectivity of $L$ \cite{SteUhl09}]\label{cor:inj}
Let $\Omega$ be strictly convex and $T>T_0$. Moreover, let $\supp f\subset \Omega$. Then $L[f]=0$ implies $f=0$. 
\end{corollary}

In the following we discuss several numerical algorithms for solving \autoref{eq:operator}.

\subsection{Landweber Iteration}
\label{sec:Landweber}
We are employing the Landweber's iteration for solving \autoref{eq:operator} and  compare it with 
 the time reversal methods presented in \autoref{subsec:TR}, which are the standard references in this field. 
More efficient iterative regularization algorithms are at hand \cite{Han95}, but these are less 
intuitive to be compared with time reversal. The Landweber algorithm reads as follows:
\begin{align}\label{eq:Landweber}
f_0:=0 \quad \text{and} \quad f_k^\delta=f_{k-1}^\delta-\omega L^*[L[f_{k-1}^\delta]-m^\delta],\quad k=1,2,\dots,
\end{align}
where $m^\delta$ stands for error-prone data with $\norm{m^\delta-m}_{L^2(\Sigma)}<\delta$, where $m=y|_\Sigma$.
For a summary of results for Landweber regularization in Photoacoustics we refer to \cite{BelGlaSch15_report} 
and the references therein.

We emphasize that Landweber's iteration converges to the \emph{minimum norm solution}
\begin{equation}
\label{eq:mp}
 f^\dagger = L^\dagger[m]\,,
\end{equation}
where $L^\dagger$ is the  Moore-Penrose inverse  (see \cite{Nas76} for a survey), 
if the data $m$ is an element of the range of $L$. This is a property which is relevant when the observation time $T$ 
is smaller than the critical time $T_0$ which guarantees injectivity of $L$.

\subsection{Time Reversal}\label{subsec:TR}

In this subsection, we first state the conventional time reversal and give a remark on necessary assumptions to obtain 
error estimates for this method. 
This is followed by a description of a modified time reversal approach, for which a theoretical analysis based on \cite{SteUhl09} can be 
provided. 

\subsubsection*{Conventional Time Reversal}\label{subsubsec:TR_TC}
We formally define the time reversal operator:
\begin{align}\label{eq:TR_operator}
\overline{L}[h] = z(\cdot,0),
\end{align} 
where $z$ is a solution of 
\begin{align}\label{eq:TR}
 \begin{aligned}
         \kappa z''- \nabla\cdot(\rho^{-1}\nabla z) &=0 \text{ in } \Omega \times (0,T),\\
         z(T) = z'(T) &=0  \text{ in } \Omega , \\         
         z  &= h \text{ on } \partial \Omega \times (0,T),\\
\end{aligned}
\end{align}
The fundamental difference between $\overline{L}$ and $L^*$ is that they are defined via differential equations 
on $\Omega \times (0,T)$ and $\R^n \times (0,T)$.
The conventional time reversal reconstruction \cite{TreCox10} consists of computing
\begin{equation}\label{eq:TR_TC}
f_\text{TR}=\overline{L}[m]\,.
\end{equation} 

\begin{remark}
Assume that $S=\partial \Omega$, that $\rho\equiv 1$ and the speed of sound is {non-trapping}.
For this case, Hristova \cite[Theorem 2]{Hri09} provides an error estimate for the time reversal method, 
employing results on the decay of solutions of the wave equation (e.g. \cite{Vai75}).
\end{remark}

\subsubsection*{Neumann Series}\label{subsubsec:TR_SU}

Stefanov and Uhlmann \cite{SteUhl09} define the modified time reversal for \autoref{eq:wave_intro}:
Rather than assuming (in most cases unjustified) the initial data $z(T) \equiv 0$ we are again using the harmonic 
extension of the data term $h(s,T)$, for $s \in \partial \Omega$, as initial datum at $T$.
That is, for 
\begin{equation*}
 - \nabla\cdot(\rho^{-1}\nabla\phi) = 0 \text{ in } \Omega, \text{ with } \phi(\cdot) = m(\cdot,T) \text{ on } \partial \Omega
\end{equation*}
the modified time reversal operator 
\begin{align}\label{eq:TR_operator_Uhl}
\widetilde{\overline{L}}[h] = z(\cdot,0)
\end{align} 
is defined by the solution of equation 
\begin{align}\label{eq:TR_Uhl}
 \begin{aligned}
         \kappa z''- \nabla\cdot(\rho^{-1} z) &=0 \text{ in } \Omega \times (0,T),\\
          z(T) = \phi,\quad z'(T) &=0 \text{ in } \Omega , \\
        z  &= h \text{ on } \partial \Omega\times(0,T).\\
\end{aligned}
\end{align}
Note that this algorithm has not been used as a basis for numerical reconstruction for the generalized 
\autoref{eq:wave_intro}.

Previously \cite{SteUhl09} showed stability of the modified time reversal reconstruction under non-trapping 
conditions and for sufficiently large measurement time for \autoref{eq:wave_old}.
In fact, let $c=(\kappa \rho)^{-1/2}$ be non-trapping \cite{HriKucNgu08} and 
$T_1$ denote the time when all
singularities have left $\overline{\Omega}$. Then the result is directly 
convertible to \autoref{eq:wave_intro}:
\begin{theorem}[Stability of Modified Time Reversal]\cite[Theorem 1]{SteUhl09}\label{thm:stability1}
Let $T>T_1$ and $S=\partial \Omega$ be a closed $C^2$-surface. 
Moreover, let the coefficients $\kappa$,~$\rho \in C^\infty(\R^n)$. 

Then $\widetilde{\overline{L}} L  = Id - K$, where $K$ is a compact 
operator from $H^1_0(\Omega) \to H^1_0(\Omega)$ satisfying $\norm{K}<1$.
\end{theorem}
Note that \cite[Theorem 1]{SteUhl09} works only for complete boundary measurement data.

By \autoref{thm:stability1}, the initial value $f$ can be expanded into the Neumann series
\begin{align}\label{eq:Neumann2}
f=\sum_{j=0}^\infty K^j[m].
\end{align}
By induction one sees that the $m$-th iterate can be written as
\begin{align}\label{eq:Neumann3}
f_k=f_{k-1}-\widetilde{\overline{L}}[L[f_{k-1}] - m],
\end{align}
where
\begin{equation*}
f_k = \sum_{j=0}^k K^j [m].
\end{equation*}
\begin{remark}
Note that with partial data, the Neumann series reconstruction consists in formally applying \autoref{eq:Neumann2}
to the extended data $m_s$, where
\begin{equation}\label{eq:data_ext}
m_s=m \text{ in }\Sigma'\subset\subset\Sigma \text{ and } m_s=0 \text{ in } \partial\Omega\setminus\Sigma\times(0,T),
\end{equation}
in a way that $m_s\in H^1(\partial\Omega\times[0,T])$ serves as approximation to $m$ in $\Sigma$.
\end{remark}

Let now $S \subset \partial \Omega$ relatively open. 
To fix the terminology, we define for $(x,\xi)\in \R^n\times \mathbb{S}^{n-1}$ the curve 
\begin{equation*}
\gamma_{(x,\xi)}(\tau)
\end{equation*} 
to be the geodesic (in the Riemannian metric $\kappa \rho \d x$) through $x$ in direction $\xi$, where $\tau\in(-\infty,\infty)$ and $\gamma_{(x,\xi)}(0)=x$. 
Note that
\begin{equation*}
\tau_{\pm}(x,\xi)=\max\set{\tau\geq 0 : \gamma_{(x,\xi)}(\pm\tau)\in\overline{\Omega}}.
\end{equation*}
give the (possibly infinite) times when the geodesics leave the domain in positive ($\sgn=+$) resp. negative ($\sgn=-$) direction.
Regarding stability \cite[Theorem 3]{SteUhl09} provides the following result, which holds also in the partial data case.
\begin{theorem}\label{thm:stability2}
Assume that
\begin{equation*}(\tau_\text{sgn}(x,\xi),\gamma_{(x,\xi)}(\tau_\text{sgn}(x,\xi))\in\Sigma\end{equation*}
holds for all $(x,\xi)\in \Omega_1 \times \mathbb{S}^{n-1}$ in at least one of the two directions $\sgn=+$, $\sgn=-$.
Then there exists a constant $C>0$ such that
for any $f\in C_0^\infty(\Omega_1)$, with $\Omega_1\subset\subset\Omega$, the estimate
\begin{align*}
\norm{f}_{H^1_0(\Omega_1)}\leq C\norm{Lf}_{H^1(\Sigma)}
\end{align*}
is valid.  
\end{theorem}

\section{Numerical Experiments and Results}
\label{sec:results}
We compare conventional time reversal \autoref{eq:TR_TC}, the Neumann series approach \autoref{eq:Neumann3} and the Landweber iteration 
\autoref{eq:Landweber} for the photoacoustic imaging problem based on 
\autoref{eq:wave_intro}. Similar studies have been performed in \cite{BelGlaSch15_report} for photoacoustic imaging 
based on \autoref{eq:wave_old}. 
\autoref{fig:parameters} displays the involved parameters, including absorption density $f$, 
material compressibility $\kappa$ and density $\rho$.

For the numerical solution of the involved wave equations \autoref{eq:wave_intro} and \autoref{eq:wave_adj}, we use a 
straight-forward adaptation of the BEM-FEM scheme outlined in \cite{BelGlaSch15_report}. 
\autoref{eq:TR} and \autoref{eq:TR_Uhl} are discretized by finite element discretization.

For the simulation of the data, for 
all different wave equations, the mesh size has been chosen as $\Delta x =0.0095$ and $\Delta t= \Delta x/(15*c_{\max})$, 
leading to about $40000$ nodal points in $\Omega$.
For all wave equations involved in reconstruction, we use a grid with $\Delta x_R = 0.01$ and $\Delta t_R=\Delta x_R/(15*c_{\max})$.

Measurement data are assumed to be recorded on $630$ detection points on the unit circle. 
In the partial data example, the measurements are restricted to the lower half of the circle.

The total measurement time was varied as multiples of $T_0$, which is defined as in \autoref{eq:T0}.
\begin{figure}[ht]\center
\includegraphics[width=0.3\textwidth]{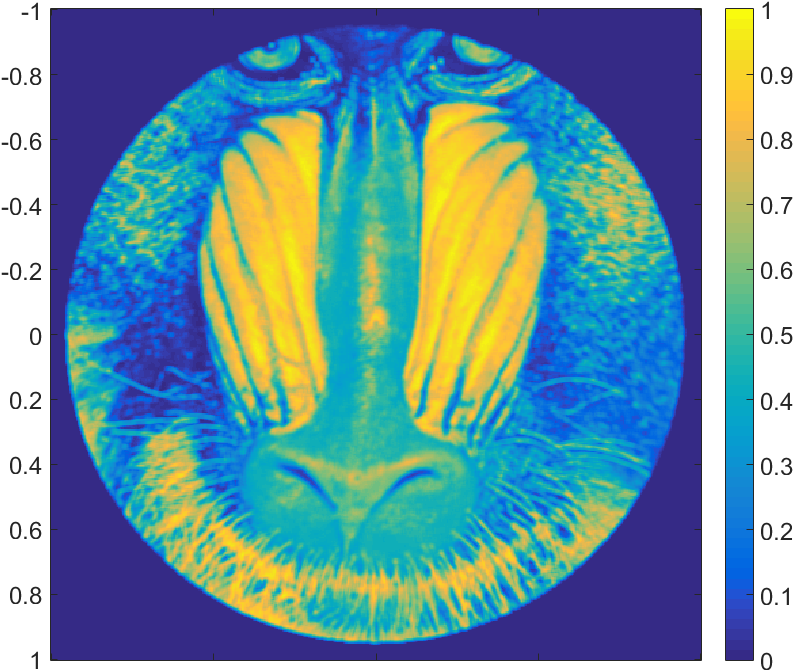}
\includegraphics[width=0.3\textwidth]{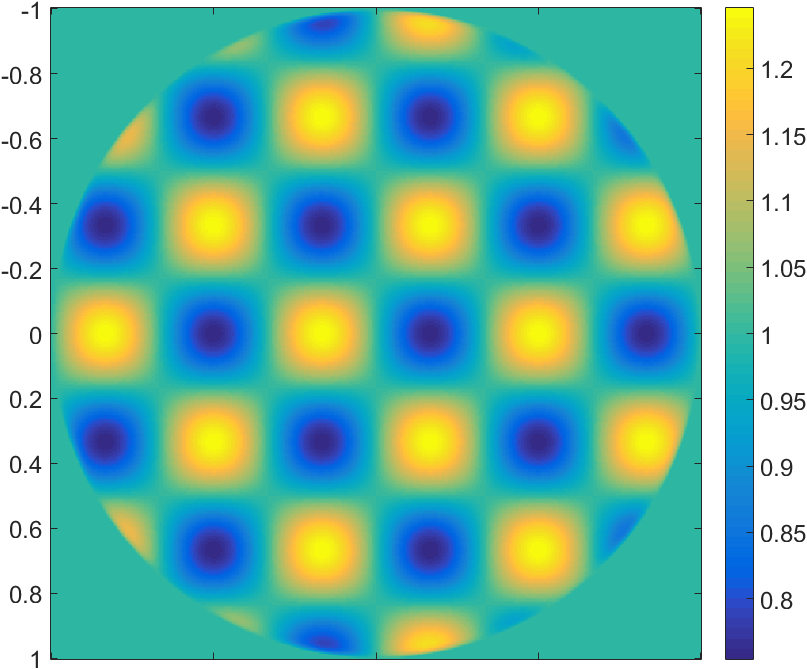}
\includegraphics[width=0.3\textwidth]{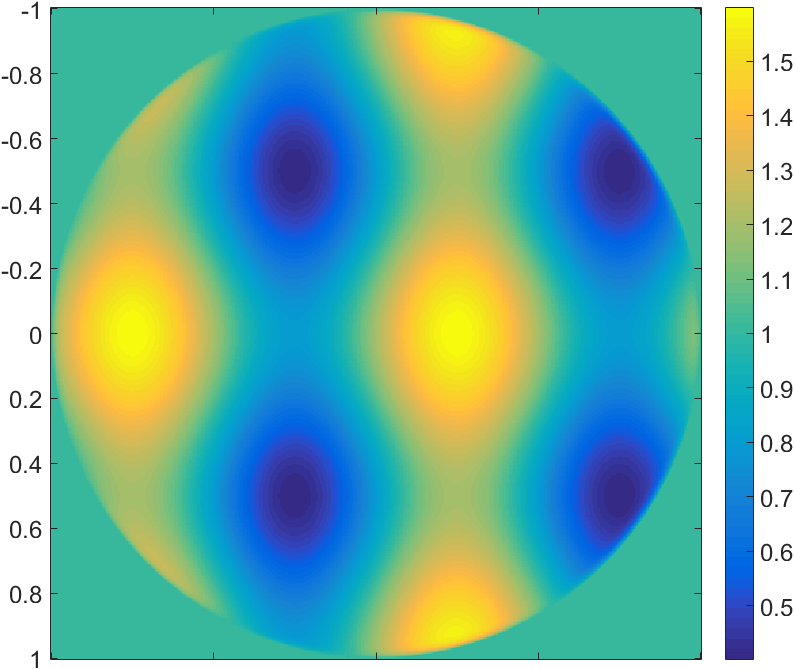}
\\
\includegraphics[width=0.3\textwidth]{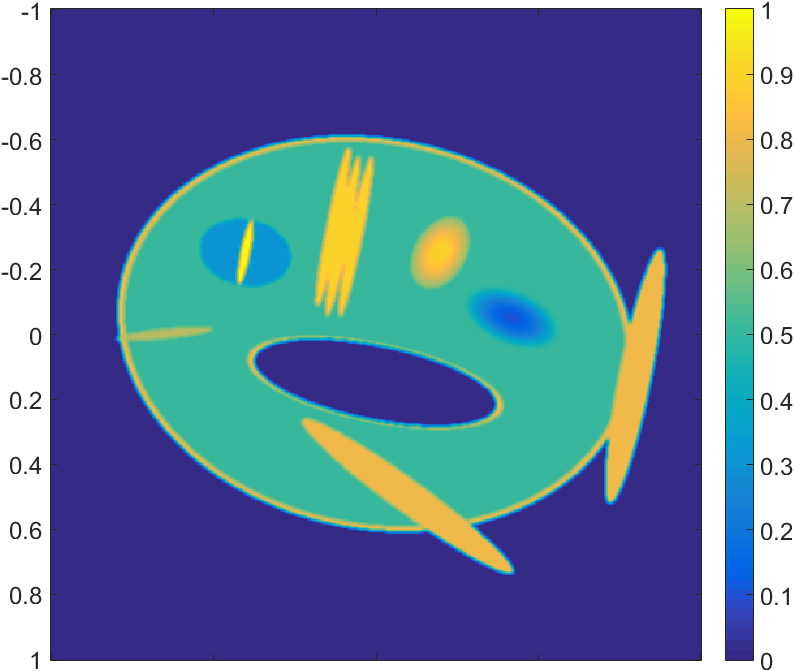}
\includegraphics[width=0.3\textwidth]{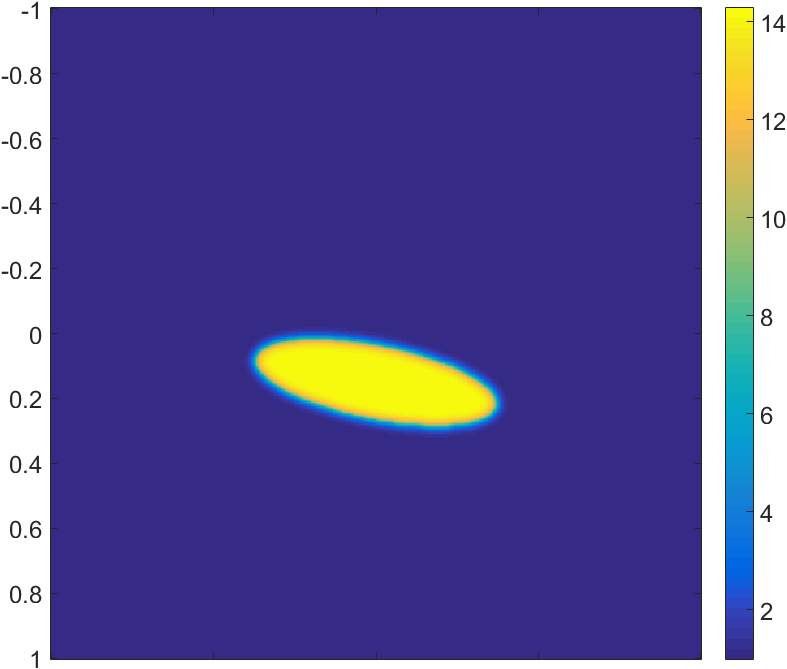}
\includegraphics[width=0.3\textwidth]{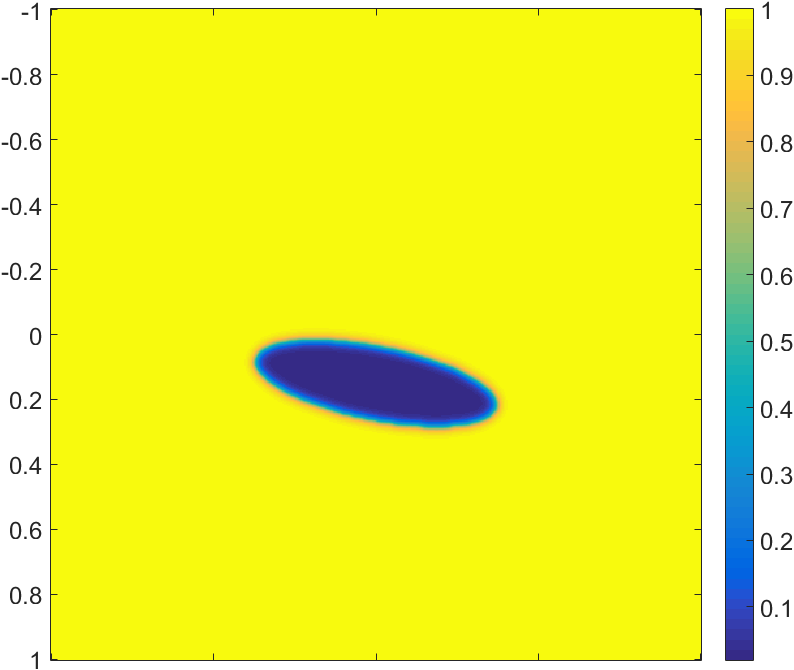}
\caption{Parameters: 
\textbf{First line:} Test Example 1 -- \emph{Mandrill}.
\textbf{Second line:} Test example 2 -- \emph{Fish}.
\textbf{From left to right:} Initial pressure $f$, compressibility $\kappa$, density $\rho$.}
\label{fig:parameters}
\end{figure}

\subsection{Test Example 1 -- \emph{Mandrill}}
We investigate the performance of the different photoacoustic imaging techniques, time reversal, Neumann series and 
Landweber iteration, respectively, with spatially varying $\kappa$ and $\rho$.
We show different test cases
\begin{enumerate}[label=\alph*]
\item
for partial measurement data (see \autoref{fig:halfcircle}), where we show the imaging results for $T=2T_0$ and $T=4T_0$,
where $T_0$, defined in \autoref{eq:T0} denotes the minimal time that guarantees unique reconstruction of the absorption density $f$
\item
as well as reconstructions from full data in the presence of noise (\autoref{fig:noise}).
The noise level is stated as SNR (signal to noise ratio, in dB scale) with respect to the maximum signal value.
The examples include moderate noise (SNR $=10$dB) and high noise (SNR $=5$dB). 
\end{enumerate}
\begin{figure}[ht]\center
\includegraphics[width=0.3\textwidth]{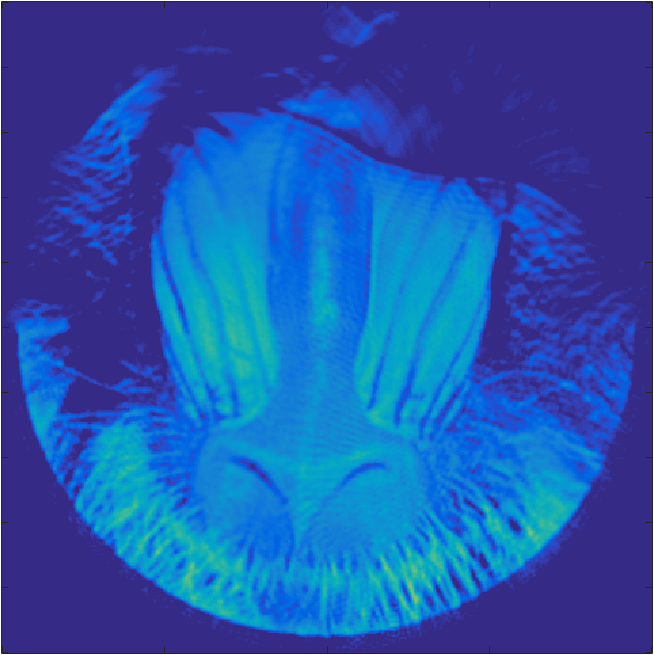}
\includegraphics[width=0.3\textwidth]{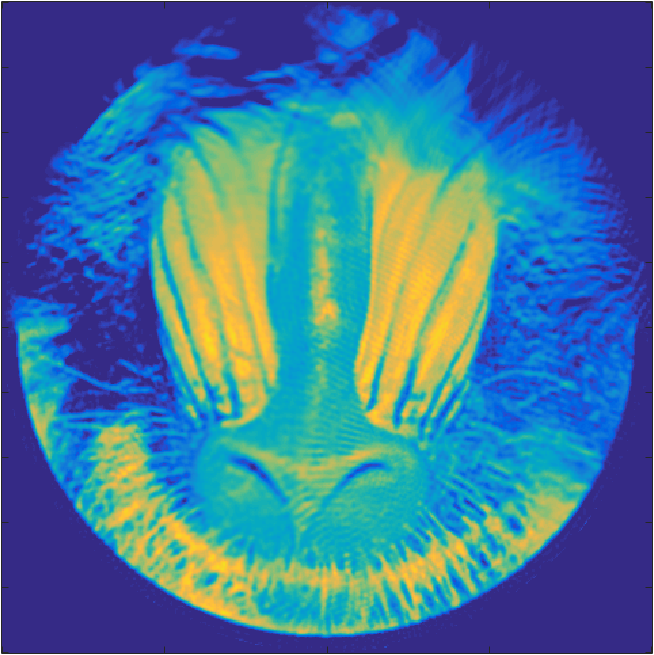}
\includegraphics[width=0.3\textwidth]{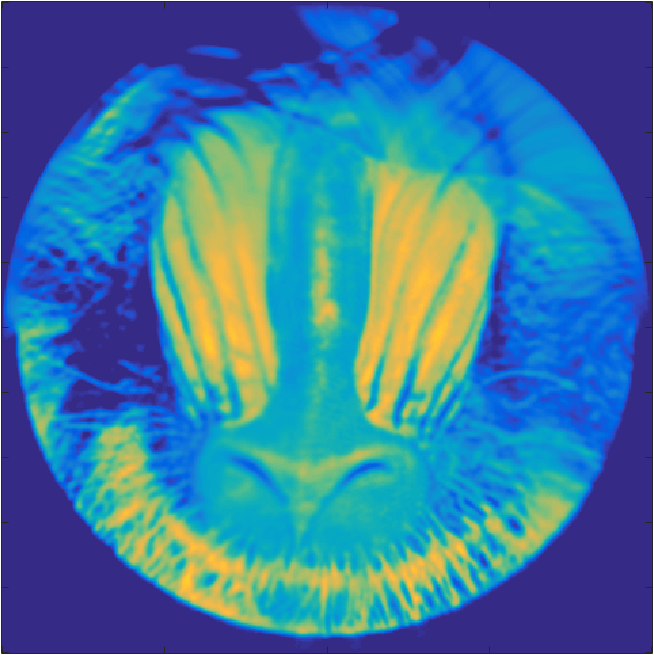}
\\
\includegraphics[width=0.3\textwidth]{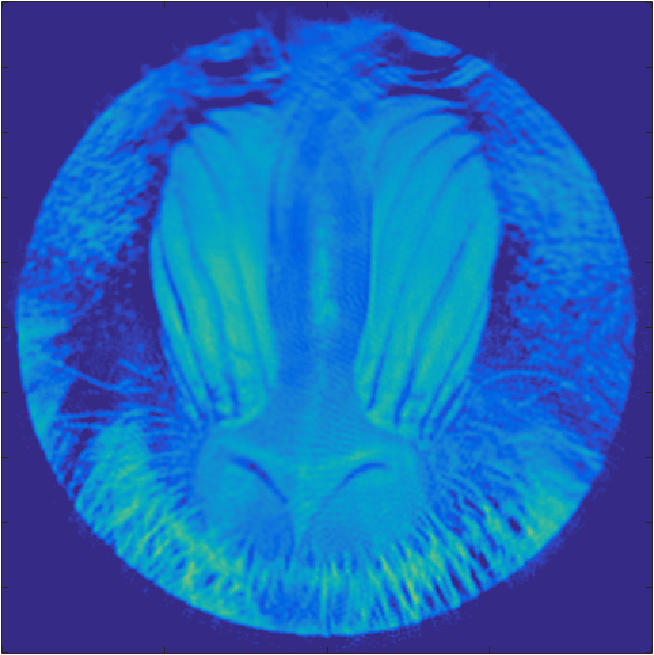}
\includegraphics[width=0.3\textwidth]{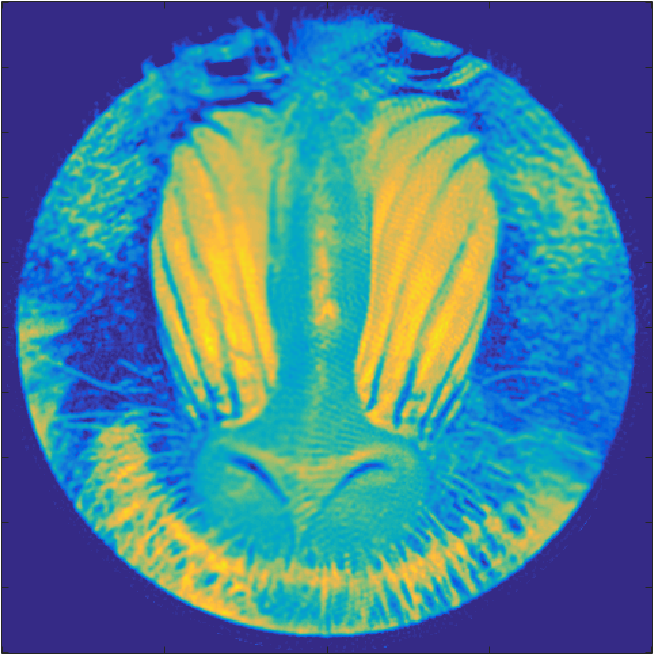}
\includegraphics[width=0.3\textwidth]{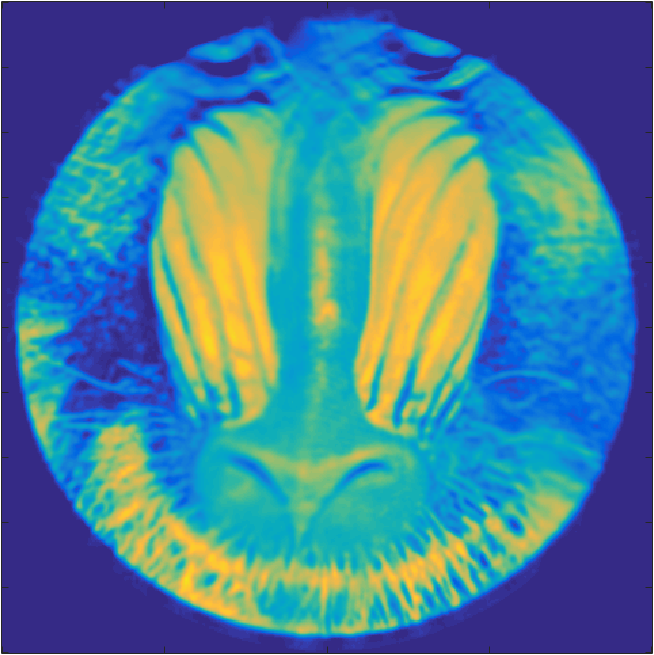}
\caption{Partial data reconstructions: Data are recorded on the lower unit half-circle and the parameters 
$\kappa$, $\rho$ are exactly given. Data are noise free and recorded for $T=2T_0$, $4T_0$.
Reconstructions for $T=2T_0$, $4T_0$ are shown in the first and second line, respectively.
\textbf{From left to right:} Time reversal, Neumann series with $k=5$, Landweber iterate $k=5$.}
\label{fig:halfcircle}
\end{figure}
\begin{figure}[ht]\center
\includegraphics[width=0.3\textwidth]{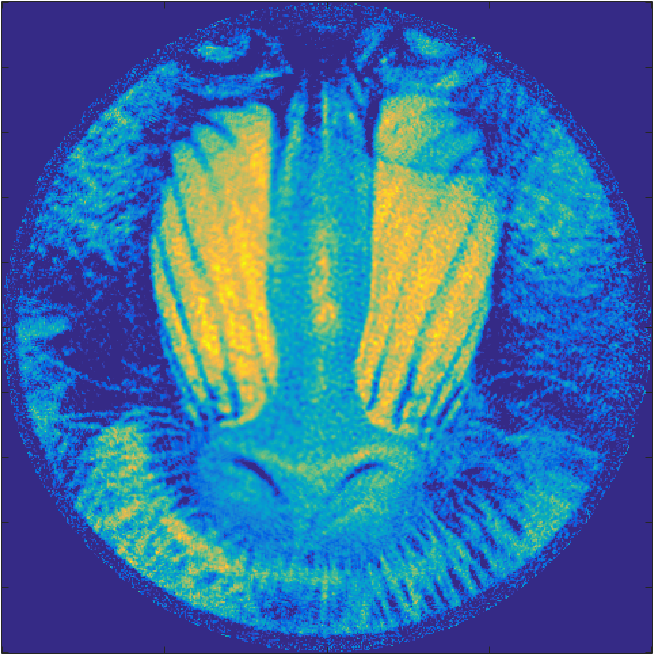}
\includegraphics[width=0.3\textwidth]{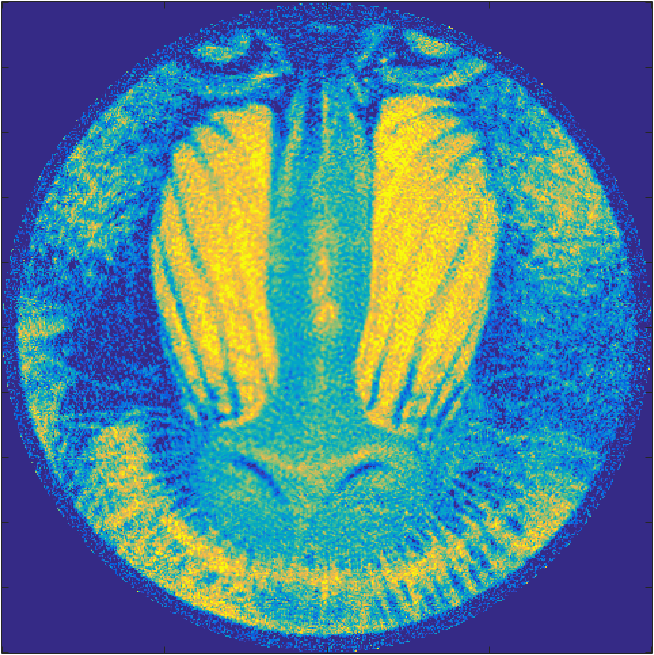}
\includegraphics[width=0.3\textwidth]{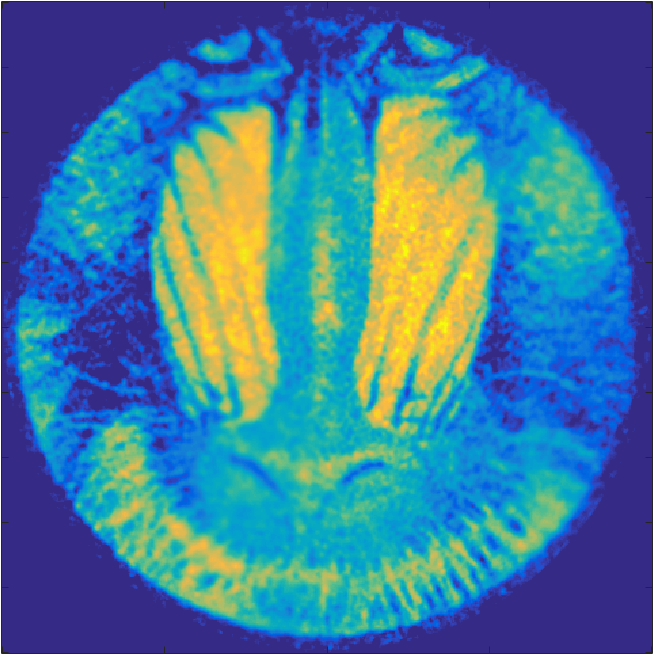}
\\
\includegraphics[width=0.3\textwidth]{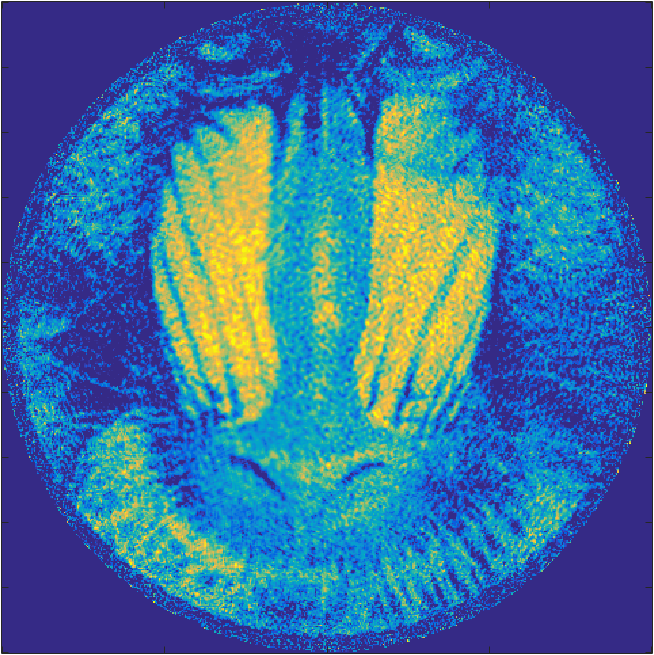}
\includegraphics[width=0.3\textwidth]{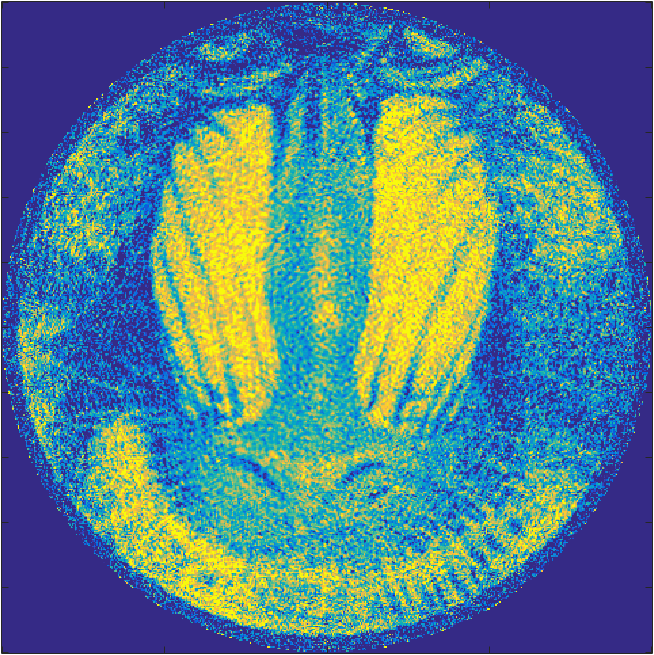}
\includegraphics[width=0.3\textwidth]{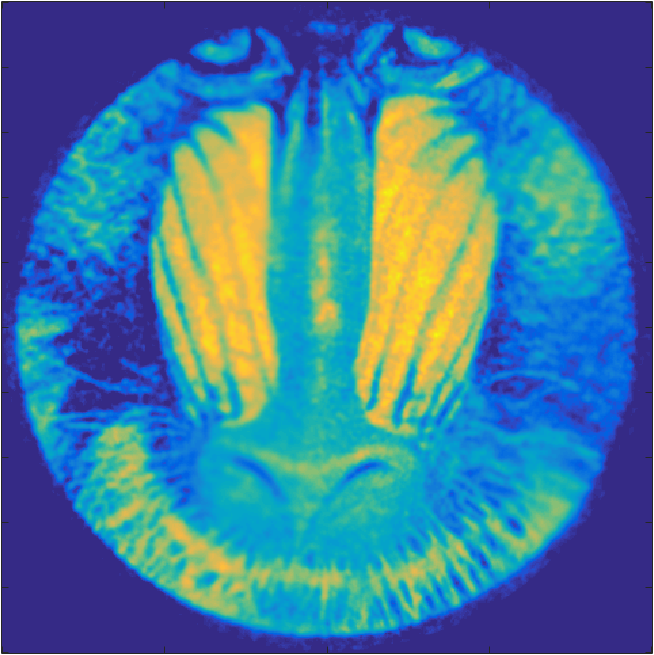}
\caption{Reconstruction with complete, but noisy measurements and correct parameters $\kappa,\rho$, 
measurement time $T=2T_0$.
\textbf{First line}: SNR $=10$dB.
\textbf{Second line}: SNR $=5$dB.
\textbf{From left to right:} Time reversal, Neumann series with $k=5$, Landweber iterate $k=5$.}
\label{fig:noise}
\end{figure}

\subsection{Test Example 2 -- \emph{Fish}}
The second test concerns a parameter setting $\kappa,\rho$ and associated sound speed $c=(\kappa \rho)^{-1/2}$, which 
is reconstructed with inversions (time reversal, Neumann series, Landweber iteration) based on \autoref{eq:wave_old}.

The phantom simulates a water-like body (e.g. soft tissue) containing an inclusion with significantly different acoustic properties, 
like the air-filled swim-bladder of a fish (see lower line of \autoref{fig:parameters}).
We choose the parameters to be bounded in the intervals $0.02\leq\rho\leq 1$ and $1\leq\kappa\leq 14.3$, 
with high gradients in the transition between the inclusion and the rest of the domain (see also \autoref{fig:parameters}).

Moreover, we include the achievable results when only $c=(\kappa \rho)^{-1/2}$ is known,
leading to a modeling error in the reconstructions.
\begin{enumerate}[label=\alph*]
\item
In the first line of \autoref{fig:fish} we display the reconstructions with correct parameters $\kappa,\rho$ 
as depicted in the second line of \autoref{fig:parameters}.
\end{enumerate}

\begin{enumerate}[resume,label=\alph*]
\item
Next, reconstructions are performed by using the parameters $\kappa_1=\kappa \rho$ and $\rho_1=1$.
This displays the usually considered approximation (see second line of \autoref{fig:fish}).
\item
We also try it the other way round by setting $\rho_2 = \kappa \rho$, and $\kappa_2\equiv 1$ (third line of \autoref{fig:fish}).
This leads to the wave equation in pure divergence form, and displays the case where the compressibility variations are negligible.
\end{enumerate} 
However, in b and c the modeling error leads to severe artifacts near regions of high-gradient regions of the parameters.

\subsection{Results}
In the presented test examples, all three methods qualitatively reconstruct the same features. 
Time reversal however fails to give a quantitatively correct results, 
due to the relatively short measurement times in use. 
Neumann series and Landweber iteration perform at the same level. 
As expected from theory, the Landweber reconstructions appear slightly smoother, specifically in \autoref{fig:noise}.

The second test clearly indicates that a modeling error in the reconstruction method can lead to severe artifacts
near regions where the parameter gradients are large.

\begin{figure}[ht]\center
\includegraphics[width=0.3\textwidth]{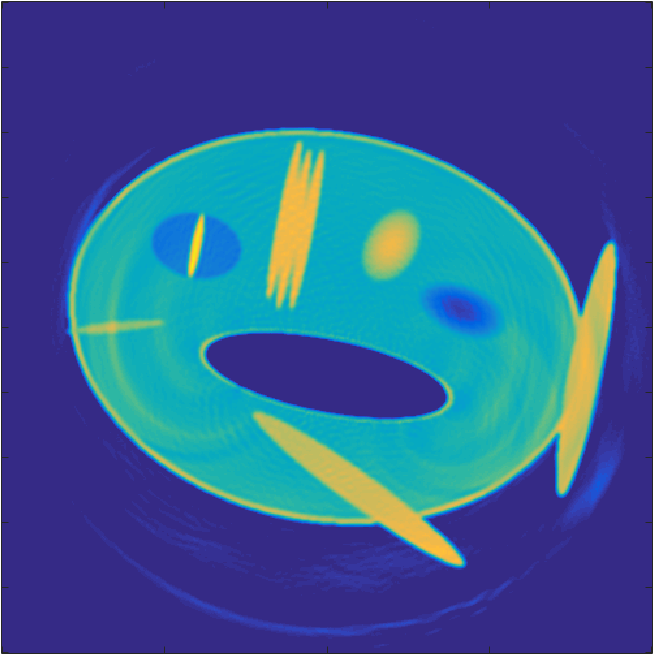}
\includegraphics[width=0.3\textwidth]{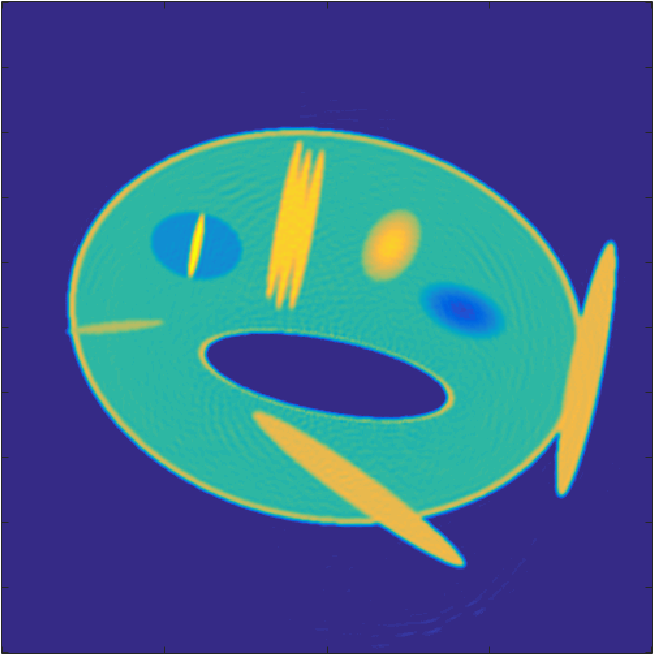}
\includegraphics[width=0.3\textwidth]{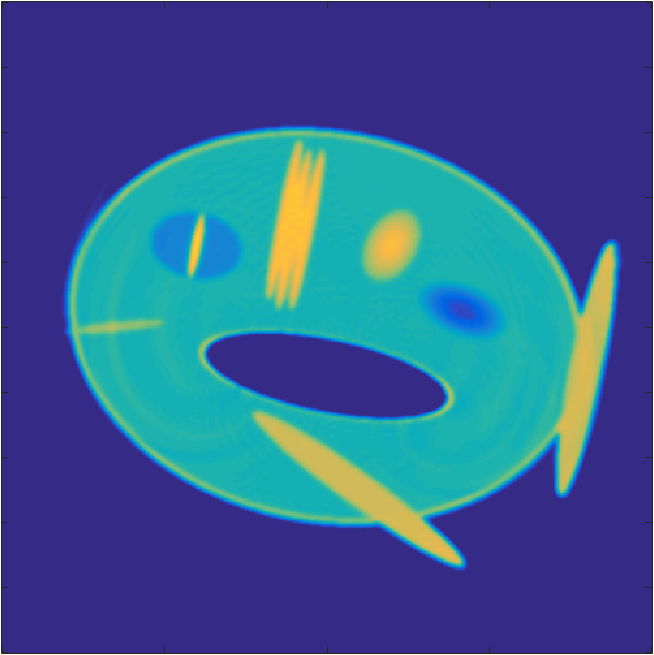}
\\
\includegraphics[width=0.3\textwidth]{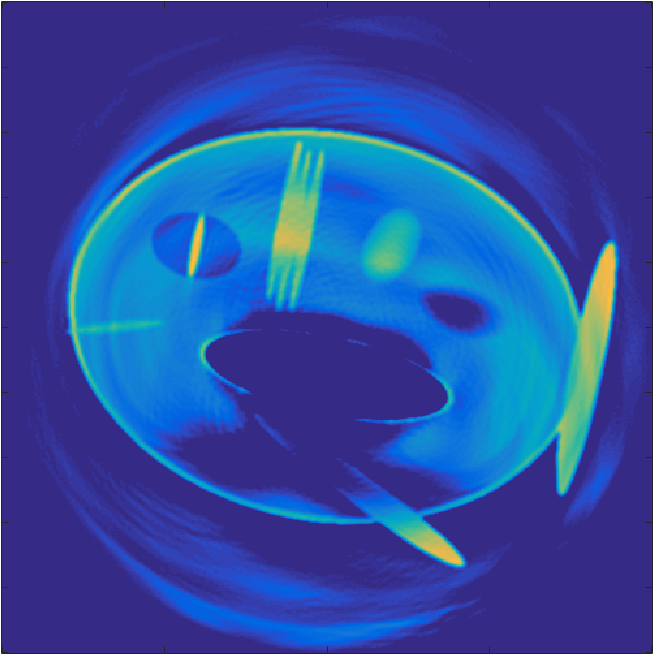}
\includegraphics[width=0.3\textwidth]{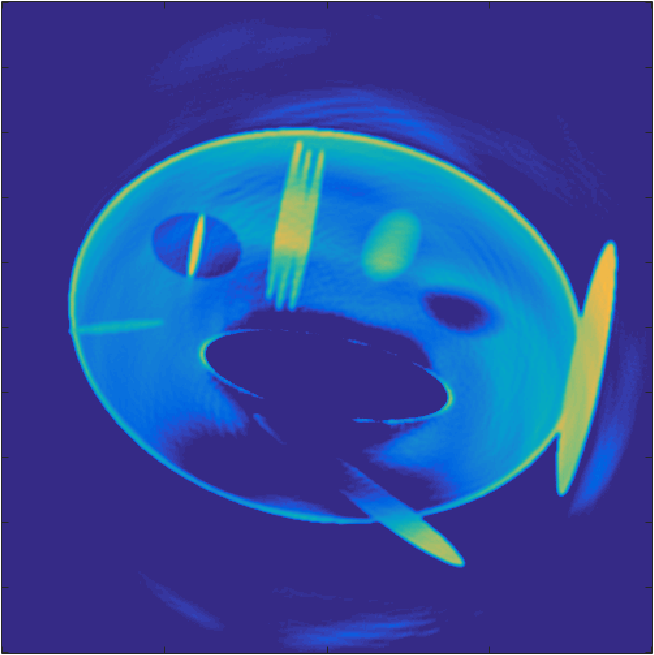}
\includegraphics[width=0.3\textwidth]{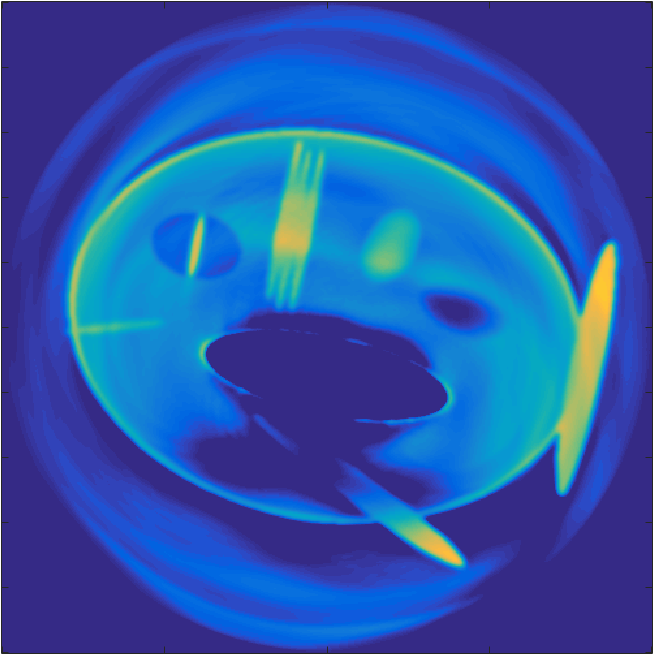}
\\
\includegraphics[width=0.3\textwidth]{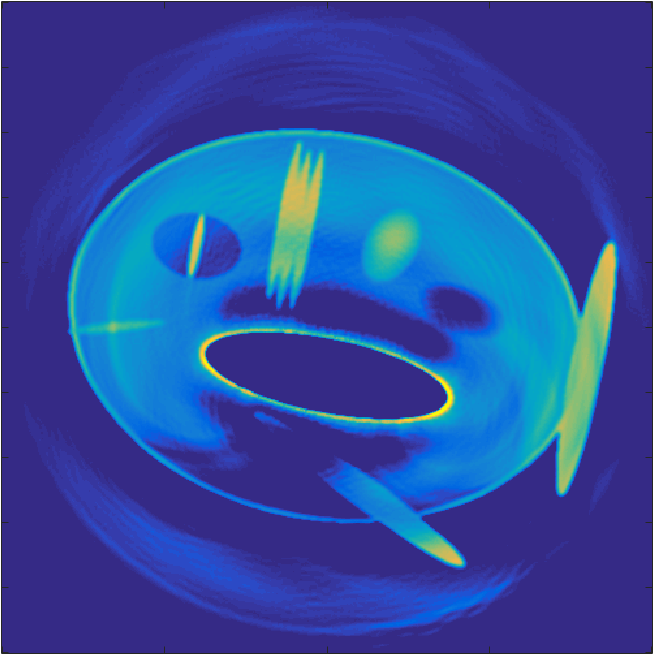}
\includegraphics[width=0.3\textwidth]{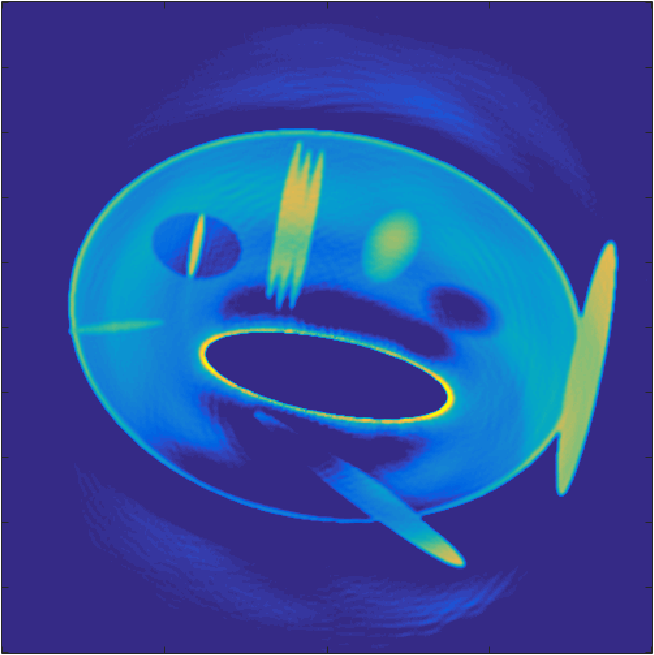}
\includegraphics[width=0.3\textwidth]{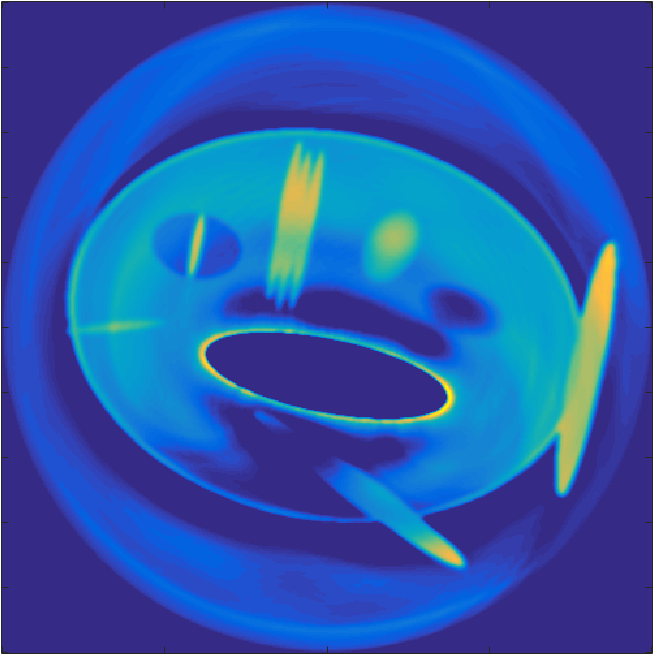}
\caption{Reconstructions of the \emph{Fish} phantom with $T=2T_0$.
\textbf{First line}: Reconstruction using correct parameters $\kappa_1=\kappa$, $\rho_1=\rho$ as pictured in the second line of \autoref{fig:parameters}.
\textbf{Second line}: Reconstruction with parameters $\kappa_1=\kappa\rho$, $\rho_1\equiv 1$.
\textbf{Third line}: Reconstruction with parameters $\rho_1=\kappa\rho$, $\kappa_1\equiv 1$.
\textbf{From left to right:} Time reversal, Neumann series with $k=5$, Landweber iterate $k=5$.}\label{fig:fish}
\end{figure}

\section{Conclusions}
In this work we have studied photoacoustic imaging based on a general wave model with spatially variable 
compressibility and density, respectively. 
We have implemented Neumann series and Landweber iteration for photoacoustic imaging 
based on this general equation and we compared the result to conventional time reversal as discussed in 
\cite{TreCox10}.

The numercial methods for photoacoustic imaging reveal the differences as outlined in \autoref{tab:methods}
with respect to convergence, stability and robustness against noise at the present stage of research. Stability 
is understood in the sense of regularization theory \cite{Gro84}, meaning that the Landweber iterates determined by 
a discrepancy principle approximate the minimum norm solution.

\begin{table}[ht]
\centering
  \label{tab:table1}
  \begin{tabular}{|l||c|c|c|}
  \hline
    &  \textbf{Time Reversal} & \textbf{Neumann} & \textbf{Landweber}\\
    \hline
    \textbf{\large Convergence:} & & & best approx. sol.\\
    measurement time & $T\to\infty$ &  $T>T_1$ & $T>0$ \\
    sound speed & non-trapping & non-trapping & arbitrary\\
    data & full & full & partial\\ 
    \hline  
    \textbf{Stability:} & & & regularized sol.\\
    measurement time & $T>T_1$ & $T>T_1/2$ & $T>0$\\
    sound speed & non-trapping & non-trapping & arbitrary\\
    data & full & partial & partial\\
    \hline
    \textbf{$L^2$-Noise:} & no & no & yes\\
    \hline
  \end{tabular}  
  \caption{Overview on the different photoacoustic imaging methods.}\label{tab:methods}
\end{table}
Numerical results show the reconstructions in the case of error prone data and under modeling errors. We emphasize that, 
so far, an error analysis is only possible for the Landweber iteration. 

\FloatBarrier
\section*{Acknowledgement}
The work of TG and OS is supported by the Austrian Science Fund (FWF), Project P26687-N25 
Interdisciplinary Coupled Physics Imaging. We are grateful to the valuable hints of Silvia Falletta concerning
details to discretization and implementation of the coupled FEM-BEM system.

%

\def\cprime{$'$}
  \providecommand{\noopsort}[1]{}\def\ocirc#1{\ifmmode\setbox0=\hbox{$#1$}\dimen0=\ht0
  \advance\dimen0 by1pt\rlap{\hbox to\wd0{\hss\raise\dimen0
  \hbox{\hskip.2em$\scriptscriptstyle\circ$}\hss}}#1\else {\accent"17 #1}\fi}

\end{document}